\documentclass[10pt]{article}

\usepackage{amssymb,amsmath,amsthm,mathtools,tabu}
\usepackage{graphicx}
\usepackage{subcaption}
\usepackage{verbatim}
\usepackage{color}
\usepackage{multirow}
\usepackage{tikz}
\usetikzlibrary{matrix,decorations.pathreplacing}
\usepackage{dsfont}
\usepackage[toc,page]{appendix}
\usepackage{tabularx}
\usepackage{cleveref}
\usepackage{csquotes}
\usepackage{indentfirst}
\usepackage{multicol}
\usepackage{setspace}
\usepackage{soul}
\usepackage[utf8]{inputenc}
\usepackage[english]{babel}

\parskip \medskipamount

\newtheorem{theorem}{Theorem}[section]
\newtheorem{prop}[theorem]{Proposition}
\newtheorem{corollary}{Corollary}[theorem]
\newtheorem{lemma}[theorem]{Lemma}

\newtheorem*{remark}{Remark}
\newtheorem{definition}{Definition}

\newtheorem{conjecture}[theorem]{Conjecture}

\newcommand{\ceil}[1]{\left\lceil #1 \right\rceil}

\newcommand{\floor}[1]{\left\lfloor #1 \right\rfloor}

\def\tmixe{t_{mix}\left(\varepsilon\right)\ }
\def\tmix{t_{mix} }
\def\tmixn{t_{mix}^{(n)}(\varepsilon)}

\def\lams{\lambda_{\star}}
\def\zp{\mathbb{Z}/p\mathbb{Z}}

\title{Mixing Times for the Commuting Chain on CA Groups}
\author{John Rahmani\\ ajrahman@usc.edu \thanks{University of Southern California, Los Angeles, California, USA}}

\begin{document}
\maketitle

\begin{abstract}

Let $G$ be a finite group. The {\it commuting chain} on $G$ moves from an element $x$ to $y$ by selecting $y$ uniformly amongst those which commute with $x$. The $t$ step transition probabilities of this chain converge to a distribution uniform on the conjugacy classes of $G$. We provide upper and lower bounds for the mixing time of this chain on a CA group (groups with a ``nice" commuting structure) and show that cutoff does not occur for many of these chains. We also provide a formula for the characteristic polynomial of the transition matrix of this chain. We apply our general results to explicitly study the chain on several sequences of groups, such as general linear groups, Heisenberg groups, and dihedral groups. 

The commuting chain is a specific case of a more general family of chains known as Burnside processes. Few instances of the Burnside processes have permitted careful analysis of mixing. We present some of the first results on mixing for the Burnside process where the state space is not fully specified (i.e not for a particular group). Our upper bound shows our chain is rapidly mixing, a topic of interest for Burnside processes.

\end{abstract}

\section{Intoduction}

The {\it commuting chain}  on a finite group $G$ is a Markov chain with state space $G$ which moves from $x$ to $y$ by selecting $y$ uniformly amongst the elements which commute with $x$. The chain converges to an equilibrium distribution uniform on the conjugacy classes of $G$. Our aim in this paper is to study the mixing times (convergence rate) of this chain when $G$ is a CA group. A group is a CA if upon removing the center commuting is a transitive relation (and thus partitions the group into equivalence classes of commuting elements). \\
\indent	Our main results are upper and lower bounds on mixing times of the commuting chain on a CA group in terms of the size of the group's center and largest non-trivial centralizer. Using our upper bound we are able to show that cutoff will not occur for this chain in many cases. We also provide a formula for characteristic polynomial for the transition matrix of this chain. Using that we see our bounds for mixing are better than what one obtains from using only the second largest eigenvalue. \\
\indent This chain is a an example of a more general family of chains known as Burnside processes introduced in \cite{MrkJOrig} (see below for more). Some initial examples were shown to have {\it rapid mixing}, meaning the mixing time is bounded by a polynomial in the size of the state space. Later work (\cite{BurnSlw}) showed this will not always be the case for a Burnside process. Our main upper bound (\Cref{minUB}) shows that the commuting chain is always rapidly mixing.\\
\indent In the remainder of this section we discuss the mixing and group theoretic preliminaries we need for our analysis, and carefully state our problem and its relation to the  Burnside process. In \Cref{bounds} we prove general bounds on the mixing time of our chain on CA groups and disprove cutoff under an additional assumption. In \Cref{cpsec} we provide a formula for the characteristic polynomial of this chain in terms of some parameters of the underlying group. We apply our general results to specific families of groups in \Cref{ex}. Finally we state some interesting features of our results and make some conjectures in \Cref{re}.

\subsection{Mixing Times}

Let $P(x,y)$ be the transition matrix of an irreducible and aperiodic Markov chain with stationary distribution $\pi$ on state space $\mathcal{X}$. It is well known that 
$\Vert P^t(x,\cdot)-\pi \Vert_{TV}\rightarrow0$ as $t\rightarrow \infty$ for any $x\in\mathcal{X}$. Where $\Vert\mu-\nu\Vert_{TV}=\sup_{A\subset\mathcal{X}}|\mu(A)-\nu(A)|$. We define 
 \[d(t)=\max_{x}\Vert P^t(x,\cdot)-\pi \Vert_{TV}\] 
 and
  \[t_{mix}(\varepsilon)=\min\{ t:d(t)\leq\varepsilon\}. \]

 We take $t_{mix}=t_{mix}\left(\frac{1}{4}\right)$ by convention.\\
 
 	 We are interested in how $\tmixe$ varies as the state space of the chain grows (for example a deck of $n$ cards as $n$ grows) so in principle $\tmixe$ depends on a size parameter, say $n$, and should be denoted $t_{mix}^{(n)}\left(\varepsilon\right)$. However we suppress the the dependence on $n$ when not needed.

\subsection{Background on Groups}

\par Given a finite group G we define the centralizer of $g\in G$ as \[C_g=\{ h\in G :hg=gh\},\]  the set of all elements which commute with $g$. The center of G is \[Z=\{h\in G: hg=gh
\text{ for all }g\in G\},\] 
the set of elements which commute with everything. The conjugacy class of an element $x\in G$ will be denoted $x^G$.\\

The orbit-stabilizer lemma (see for instance \cite{DF}) tells us $|C_x||x^G|=|G|$.\\

We will work with the following class of groups. 
\begin{definition} A group G is a CA (or CT) group if commutativity is a transitive relation on $G\setminus Z$. \end{definition}

\begin{remark}

A CA group is partitioned into the center and disjoint sets of elements which commute. An alternate definition of CA group is that the centralizer of any non-central element is abelian.  

\end{remark}

CA groups with a trivial center have been classified. The following result can be found in ~\cite{CTgroupsnew}. 
\begin{theorem}
Every non-abelian simple CA group with trivial center is isomorphic to some $PSL(2,2^k), k\geq2$. 
\end{theorem}

Here $PSL(2,q)$ is the $2\times2$ projective special linear group over a field of order $q$, this is the group obtained by taking the quotient group of $SL(n,q)$ by the subgroup of scalar matrices with unit determinate.

\subsection{Commuting Chain and the Burnside process }

\par Suppose $G$ is a finite group. The {\it commuting chain} on $G$ is a Markov chain with state space $G$ and transition probabilities 

\[P(x,y)= \frac{1}{|C_x|}\mathds{1}_{\{xy=yx\}}\text{ for all }x,y\in G. \] The chain moves from $x$ by picking the next state uniformly at random from those which commute with $x$.\\

Since the identity commutes with everything the chain is irreducible, and since any element commutes with itself the chain is aperiodic. \\

\par For $x\in G$ let  
 \[\pi(x)=\frac{1}{k|x^G| }\]

where $k$ is the number of conjugacy classes of $G$. $k$ normalizes $\pi$ as $\sum_{x\in G} \frac{1}{|x^G|}=k$ since each $x$ appears in only one conjugacy class.\\

The orbit stabilizer lemma tells us that 
\[\pi(x)P(x,y)=\frac{1}{k|x^G|} \frac{1}{|C_x|}\mathds{1}_{\{xy=yx\}}=\frac{1}{k|y^G|} \frac{1}{|C_y|}\mathds{1}_{\{xy=yx\}}=\pi(y)P(y,x)\]

and so the commuting chain is reversible with respect to $\pi$. Irreducibility and aperiodicity then ensure that the $t$-step transition probabilities converge to $\pi$ in total variation.  Our goal is to bound the mixing times of this chain for CA groups.\\
\indent The commuting chain is a special case of the {\it Burnside process} introduced in \cite{MrkJOrig}. The Burnside process takes place in the more general context of a group acting on a set. Let $G$ be a group acting on $\mathcal{A}$. The  Burnside process has state space $\mathcal{A}$ and moves from $a$ to $b$ as follows. From $a$ select uniformly a $g\in G$ such that $g\cdot a=a$. Now given $g$ select uniformly amongst the $b$ such that $g\cdot b=b$. The corresponding transition matrix is reversible with respect to a measure uniform on the orbits of the group action. The commuting chain can be seen as a (interpolated) Burnside process where a group is acting on itself through conjugation. \\
\indent The earliest analysis of a case of the Burnside process (by Jerrum in ~\cite{MrkJOrig}) showed a chain with the {\it rapid mixing} property - that is the mixing time is bounded by a polynomial in the size of the state space. In \cite{BurnSlw}  Goldberg and Jerrum construct a Burnside process which does not mix rapidly. \Cref{minUB} shows that the commuting chain does mix rapidly.\\
\indent The  results in \cite{BurnSlw} were proven by comparing the chain to a Swendsen-Wang algorithm. Swendsen-Wang algorithm is a graph coloring Markov chain from statistical mechanics which converges to a Potts model. Swendsen-Wang with a two color Potts model is a common alternative to the Glauber dynamics, as they both converge to the same Gibbs distribution and in many cases Swendsen-Wang will converge faster than the Glauber dynamics.\\
\indent In \cite{AFILL} Aldous and Fill study mixing times for a Burnside process with a coupling. In \cite{DiacBos} Diaconis bounds mixing times for the chain studied by Aldous and Fill, using a minorization condition for the upper bound. In \cite{HitNRun} the commuting chain is listed as an example of the Burnside process in the context of Markov chains which converge quickly. Chapter two of \cite{permgroup} also mentions the chain and its connections to permutation groups.  \\
 \indent These are the only examples we know of where the Burnside process has been studied. In particular we believe this paper presents the first (published) mixings bounds for the commuting chain.

\subsection{Minorization}

 We use this basic form of a minorization bound that can be found in \cite{MinSurv}.
 
 \begin{theorem}\label{minorization}
Let $P$ be the transition matrix for an irreducible, aperiodic Markov chain with stationarity distribution $\pi$ on state space $\mathcal{X}$. Let $Q$ be a probability measure on $\mathcal{X}$ such that $Q(A)>0$ whenever $\pi(A)>0$. Suppose for some $1>\delta>0$ we have that 
\[P^{t_0}(x,A)\geq\delta  Q(A) \text{ for all } x \in G,A\subset \mathcal{X} \]

then 

\[d(t)\leq(1-\delta)^{ \floor{\frac{t} {t_0} }}.\]

\end{theorem}

The bounds attained are not always suitable for mixing problems. Two examples where minorization has been useful are in \cite{DiacNonRev} and \cite{DiacBos}, that latter of which was studying a  Burnside process.\\
\indent The reason we can use these bounds work well for us is that our chain makes ``big jumps" - that is we have an allowable transition between any two states with at-most $2$ steps.

\subsection{Coupling}

We present a simple use of coupling to bound mixing times, see chapters  5,14 of \cite{LPW} for more. 

\begin{theorem}\label{PthCou}
Let $\mathcal{X}$ be the state space of a finite ergodic Markov chain, and $\rho$ be a metric on $\mathcal{X}$ satisfying $\rho(x,y)\geq\mathds{1}_{\{x\neq y\}}$. Suppose there exists a constant $\alpha >0$ and a coupling $(X_1,Y_1)$ of the Markov chain satisfying 

\[\mathbb{E}_{x,y}[\rho(X_1,Y_1)]\leq e^{-\alpha}\rho(x,y)\] 

for all $x,y\in\mathcal{X}$.  Then 

\[t_{mix}(\varepsilon)\leq \ceil{\frac{1}{\alpha}\log\left(\frac{\Delta}{\varepsilon}\right)}\]

for $\Delta = \max_{x,y}\rho(x,y)$. 

\end{theorem}

\subsection{Eigenvalues}

It is well known that the  $n\times n$ transition matrix of a reversible Markov chain has  $n$ real eigenvalues corresponding to real valued eigenfunctions. For an irreducible, aperiodic, reversible chain let $\lambda_i$ for $1\leq i\leq n$ be the eigenvalues ordered as $1=\lambda_1>\lambda_2\geq\hdots\geq\lambda_n>-1$. Let $\lams=\max\{\lambda_2,|\lambda_n|\}$, the relaxation time of the chain is then $t_{rel}=\frac{1}{1-\lams}$. The following bound can be found in \cite{LPW}

\begin{prop}\label{StdEigBnd}
Let $\tmixe$ be the mixing time for a reversible, irreducible, and aperiodic Markov chain with stationary distribution $\pi$  and relaxation time $t_{rel}$. Then

\[\left(t_{rel}-1\right)\log{\left(\frac{1}{2\varepsilon}\right)}\leq\tmixe\leq \ceil{ t_{rel} \log{\left(\frac{1}{\varepsilon \pi_{min}} \right)}}\]

where $\pi_{min}=\min_{x}\pi(x)$.

\end{prop}

\subsection{Cheegar Constant}

Suppose $P$ is the transition matrix of an irreducible, aperiodic Markov chain with stationary distribution $\pi$ on $\mathcal{X}$. For $S\subset \mathcal{X}$ let $\Phi(S)=\frac{\sum_{x\in S}\sum_{y\in S^c}\pi(x)P(x,y)}{\pi(S)}$ and \[\Phi_\star=\min_{S:\pi(S)\leq1/2}\Phi(S,S^c).\]

$\Phi_\star$ is often called the Cheegar constant or bottleneck ratio. $\Phi_\star$ is most often used to bound eigenvalues of a reversible transition matrix, but can also be used to directly bound mixing times. The following results can be found in \cite{LPW} .

\begin{theorem}\label{sgapbnd} 
For a reversible Markov chain let $\Phi_\star$ be as above. Suppose $\lambda_2$  is the second largest eigenvalue of the transition matrix of the Markov chain. Then we have 
\[\frac{\Phi_\star^2}{2}\leq 1-\lambda_2 \leq2\Phi_\star.\]
\end{theorem}

\begin{theorem}\label{CLB} 
Let $\Phi_\star$ be as above and $\tmix$ be the mixing time for the corresponding Markov chain. Then we have
 \[\frac{1}{4\Phi_{\star}}\leq t_{mix}.\]
\end{theorem}

\subsection{Cutoff}

Suppose $t^{(n)}_{mix}\left(\varepsilon\right)$ are the mixing times for a sequence of Markov chains. We say the chain exhibits a {\bf cutoff} if 

\[\lim_{n\to\infty}\frac{t^{(n)}_{mix}\left(\varepsilon\right)}{t^{(n)}_{mix}\left(1-\varepsilon\right)}=1\]

for all $\varepsilon\in(0,1)$. A necessary condition for cut off is 

\begin{prop}\label{ProdCnd}
Let $t^{(n)}_{mix}$ and $t^{(n)}_{rel}$ be the mixing times and relaxation times (respectively) for a sequence of reversible ,irreducible, aperiodic Markov chains. Suppose that $\frac{t^{(n)}_{mix}}{t^{(n)}_{rel}-1}$ is bounded above. Then there is no cutoff.

\end{prop}

Much work has gone into proving $(1-\lambda^{(n)}_\star)t^{(n)}_{mix}\to\infty$ (referred to as the product condition) is necessary and sufficient for cuttoff in various families, see \cite{DiacSalBD},\cite{BDTV}. In \cite{CutoffChenSal} the product condition was shown to be necessary and sufficient for cutoff for reversible chains when distance to stationarity is measured in $L_p$ for $p>1$ (total variation corresponds to $p=1$). In \cite{CutoffChar} a characterization of when the product condition is equivalent to cutoff is given in terms of hitting times for reversible lazy chains.\\

\section{Bounds for $\tmix$ on CA groups}\label{bounds}

\subsection{A Lower Bound}

\begin{theorem}\label{LB}

Suppose $G$ is a CA group with center $Z$ of size $z$ and $j$ centralizers $C_i$ $1\leq i\leq j$ disjoint apart from the center. Suppose $\pi$ is the stationary measure for the commuting chain on $G$. If $\pi(C_i\setminus Z)\leq\frac{1}{2}$ then for the commuting chain on $G$ we have 

\[\frac{c_i}{4z}\leq t_{mix}\]

where $c_i=|C_i|$ 

\end{theorem}

\begin{proof}
Suppose $C_i$ is a centralizer with $\pi(C_i\setminus Z)\leq\frac{1}{2}$. Set $S=(C_i\setminus Z)$. \\

We follow the notation established in \Cref{CLB}. For a fixed $x\in S$ there are $z$ elements in $S^c$ which commute with $x$ (the central elements). We see that
\[Q(S,S^c)=\sum_{x\in S}\sum_{y\in S^c}\pi(x)P(x,y)=\sum_{x\in S}\sum_{y\in S^c}\frac{1}{k|G|}\mathds{1}_{\{xy=yx\}}\]
\[=\frac{(c_i-z)z}{k|G|}.\]

And
\[\pi(S)=\sum_{x\in S}\pi(x)=\sum_{x\in S}\frac{1}{k|x^G|}=\sum_{x\in S}\frac{c_i}{k|G|}=\frac{c_i(c_i-z)}{k|G|}.\]

 so $\Phi_\star\leq \frac{z}{c_i}$. The conclusion follows from \Cref{CLB}.

\end{proof}

\subsection{Upper bound via Minorization}

The following bound does note require the group to be CA.

\begin{theorem}\label{minUB}
Let $G$ be a finite non-abelian group with center $Z$ of size $z$. For the commuting chain on $G$ we have 
\[t_{mix}(\varepsilon)\leq \ceil{\frac{2  c_\star }{ z }\log\left(\frac{1}{\varepsilon}\right)+2}\]

where $c_\star=\max\{|C_x|:x\notin Z\}$

\end{theorem}

\begin{remark}
Since $\frac{c_\star}{z}\leq |G|$ this shows that the commuting chain is rapid mixing.
\end{remark}

We need the following Algebraic lemma
\begin{lemma}\label{centbnd}
Let $G$ be a finite non-abelian group with center $Z$. Then we have
\begin{itemize}
\item $\frac{|G|}{|Z|}\geq 2$
\item if $|C|$ is the centralizer of a non-central element then $\frac{|C|}{|Z|}\geq 2$.
\end{itemize}
\end{lemma}
Both these assertions follow from Lagrange's Theorem, which states (since $Z$ is a subgroup of $G$ and of $C$) that $\frac{|G|}{|Z|}$ and $\frac{|C|}{|Z|}$ are integers. Since $G$ is non-abelian it must be that $C\neq Z\neq G$.

\begin{proof}[Proof of \Cref{minUB}]

Let $P$ be the transition matrix for the commuting chain on G. We first show

\begin{equation}\label{eq:2stpbnd}P^2(x,y)\geq\frac{z}{|G| c_\star } \end{equation}

for all $x,y\in G$. Let $c'=\min\{|C_x|:x\notin Z\}.$ We bound $P^2(x,y)$ by considering the following cases: 

\begin{enumerate}
\item $x$ is non-central and does not commutes with $y$\label{nn} 
\item $x$ is non-central and commutes with non-central $y$ \label{nnc}
\item $x$ is non-central and $y$ is central \label{nc}
\item $x$ and $y$ are central \label{cc}
\item $x$ is central and $y$ is not \label{cn}.\\
\end{enumerate}

For case \ref{nn} note transition from $x$ to $y$ can occur only by transition to $Z$ in-between and so\[P^2(x,y)=\sum_{g\in Z}P(x,g)P(g,y)=\sum_{g\in Z}\frac{1}{|C_x|}\frac{1}{|G|}\geq\sum_{g\in Z}\frac{1}{ c_\star }\frac{1}{|G|}=\frac{z}{ c_\star |G|}.\]

 For case \ref{nnc} we have
\[P^2(x,y)=\sum_{g\in Z}P(x,g)P(g,y)+\sum_{g\in C_y\setminus Z}P(x,g)P(g,y)=\sum_{g\in Z}\frac{1}{|C_x|}\frac{1}{|G|}+\sum_{g\in C_y\setminus Z}\frac{1}{|C_x||C_y|}\]
\[\geq\frac{z}{|G| c_\star }+\frac{c'-z}{\left(c_\star\right)^2}\geq\frac{z}{|G| c_\star }.\]

Similarly for case \ref{nc} 
\[P^2(x,y)=\sum_{g\in Z}P(x,g)P(g,y)+\sum_{g\in C_x\setminus Z}P(x,g)P(g,y)=\sum_{g\in Z}\frac{1}{|C_x|}\frac{1}{|G|}+\sum_{g\in C_x\setminus Z}\frac{1}{|C_x|^2}\]
\[\geq\frac{z}{|G| c_\star }+\frac{c'-z}{\left(c_\star\right)^2}\geq\frac{z}{|G| c_\star }.\]

For case \ref{cc} we have
 \[P^2(x,y)=\sum_{g\in Z}P(x,g)P(g,y)+\sum_{g\notin Z}P(x,g)P(g,y)=\sum_{g\in Z}\frac{1}{|G|^2}+\sum_{g\notin Z}\frac{1}{|G|}\frac{1}{|C_g|}\]
 
\[\geq \frac{z}{|G|^2}+\sum_{g\notin Z}\frac{1}{|G|}\frac{1}{ c_\star }=\frac{z}{|G|^2}+\frac{|G|-z}{|G| c_\star }\geq\frac{|G|-z}{|G| c_\star }\geq \frac{z}{|G| c_\star }\]
the last inequality comes from the fact that $|G|\geq 2z$ which follows from \cref{centbnd}.

Finally for case \ref{cn} \[P^2(x,y)=\sum_{g\in Z}P(x,g)P(g,y) + \sum_{ g\in C_y\setminus Z}P(x,g)P(g,y)=\sum_{g\in Z}\frac{1}{|G|^2}+\sum_{g\in C_y\setminus Z}\frac{1}{|G|}\frac{1}{|C_y|}\]
\[\geq\frac{z}{|G|^2}+\sum_{g\in C_y\setminus Z}\frac{1}{|G|}\frac{1}{ c_\star }\geq \frac{z}{|G|^2}+\frac{c'-z}{|G| c_\star }\geq\frac{z}{|G| c_\star }\]

where last inequality follows from \cref{centbnd}.

And so \cref{eq:2stpbnd} holds. Now take $Q$ to be the uniform distribution on $G$ and $\delta=\frac{z}{ c_\star }$. Then \[P^2(x,A)\geq|A|\frac{z}{|G| c_\star }=\delta Q(A)\] for any $A\subset G$ and $x\in G$, so from \Cref{minorization} we have

\[d(t)\leq\left(1-\frac{z}{ c_\star }\right)^{ \floor{\frac{t} {2} }}\leq \left(1-\frac{z}{ c_\star }\right)^{ \frac{t} {2}-1 }\leq e^{-\frac{z}{ c_\star }\left(\frac{t}{2}-1\right)}\]

and so 
\[t_{mix}(\varepsilon)\leq\ceil{\frac{2  c_\star }{ z }\log\left(\frac{1}{\varepsilon}\right)+2}.\]\\
\end{proof}

\begin{corollary}\label{cutoff}
Let $G^{(n)}$ be a sequence of finite CA groups. Let  $c^{(n)}_\star=\max\{|C^{(n)}_x|:x\notin Z^{(n)}\}$ where $Z^{(n)}$ is the center of $G^{(n)}$. Take $C^{(n)}_\star$ to be any of the centralizers which attains size $c^{(n)}_\star$. Let $\pi^{(n)}$ be the stationarity distribution for the commuting chain on $G^{(n)}$. Let $\lams^{(n)}$ be the (absolute) second largest eigenvalue of the transition matrix. If $\pi^{(n)}(C^{(n)}_\star\setminus Z^{(n)})\leq\frac{1}{2}$ and $\lams^{(n)}\rightarrow 1$ as $n\rightarrow \infty $ then the commuting chains on $G^{(n)}$ do not present cutoff. 

\end{corollary}

\begin{proof}

If $\pi^{(n)}(C^{(n)}_\star\setminus Z^{(n)})\leq\frac{1}{2}$ we know $\Phi_\star^{(n)}\leq\frac{z^{(n)}}{c^{(n)}_\star}$ (from the calculations in the proof of \Cref{LB}). Since $\lambda^{(n)}_2\leq\lams^{(n)}$ \Cref{sgapbnd} says

\[1-\lams^{(n)}\leq 1-\lambda^{(n)}_2\leq 2\frac{z^{(n)}}{c^{(n)}_\star}\]

 Combined with \Cref{minUB} we have 

\[\frac{t^{(n)}_{mix}}{t^{(n)}_{rel}-1}=\frac{1-\lams^{(n)}}{\lams^{(n)}}\tmix^{(n)}\leq \frac{1}{\lams^{(n)}}\left(4\log\left(4\right)+\frac{4z^{(n)}}{c^{(n)}_\star}\right)\leq \frac{1}{\lams^{(n)}}\left(4\log\left(4\right)+4\right) \] 


since $\frac{z}{c_\star}\leq1$. The product condition (\Cref{ProdCnd}) and our assumptions on $\lams^{(n)}$ ensures there is no cutoff. \\

\end{proof}

\subsection{Upper bound via Coupling}

In the case where all the centralizers of non-central elements are the same size a simple coupling gives a good upper bound.

\begin{theorem}\label{UCAmixing}

Let $G$ be a $CA$ group with center of size $z$ and all centralizers (of non-central elements) of size $c$. Then for the commuting chain on $G$ we have \[\tmixe \leq\ceil{\frac{1}{ \alpha }\log\left(\frac{1}{\varepsilon}\right)}\]

where $\alpha=\min\{\frac{z}{c},\frac{c}{|G|} \}$

\end{theorem}

\begin{proof}

let $\rho(x,y)=\mathds{1}_{\{x=y\}}$ be the discrete metric on $G$. We present a coupling in four cases for the initial pair of states $(x_0,y_0)$.

\begin{enumerate}

\item $x_0,y_0$ are both central elements
\item $x_0,y_0$ are elements which commute and are non-central
\item $x_0$ is central while $y_0$ is not
\item $x_0$ and $y_0$ do not commute.

\end{enumerate}

In cases $1$ and $2$ moving $X$ and $Y$ to the same element produces a coupling. In both cases $\mathbb{E}_{x_0,y_0}[\rho\left(X_1,Y_1\right)]=0$.\\

For case $3$ move $Y_1$ as a usual commuting chain. With probability $\frac{c}{|G|}$ move $X_1$ to $Y_1$, otherwise move $X_1$ uniformly amongst  the $|G|-c$ elements which do not commute with $y_0$. So we have $P(X_1=z)=\frac{1}{c}\frac{c}{|G|}=\frac{1}{|G|}$ if $z$ commutes with $y_0$, If not we have $P(X_1=z)=\left(1-\frac{c}{|G|}\right)\frac{1}{|G|-c}=\frac{1}{|G|}$. Here we have 

\[\mathbb{E}_{x_0,y_0}\left[\rho\left(X_1,Y_1\right)\right]=\mathbb{P}\left(X_1\neq Y_1\right)=1-\frac{c}{|G|}\leq e^{-\frac{c}{|G|}}.\]\\

	In case 4 move $x_0$ as usual. If $X_1$ is central set $Y_1=X_1$, if not select $Y_1$ uniformly amongst the $c-z$ non-central elements that commute with $y_0$. We have $\mathbb{P}\left(Y_1=z\right) =\frac{1}{c}$ for central $z$ and $\mathbb{P}\left(Y_1=z\right)=\left(1-\frac{z}{c}\right)\frac{1}{c-z} =\frac{1}{c}$ for non-central $z$. For the expectation we have

\[\mathbb{E}_{x_0,y_0}\left[\rho\left(X_1,Y_1\right)\right]=\mathbb{P}\left(X_1\neq Y_1\right)=1-\frac{z}{c} \leq e^{-\frac{z}{c}} .\]

Taking $\alpha=\min\{\frac{z}{c},\frac{c}{|G|} \}$ ensures $\mathbb{E}_{x,y}[\rho(X_1,Y_1)]\leq e^{-\alpha}$ for all $x_0,y_0$ and so by \Cref{PthCou} $t_{mix}(\varepsilon)\leq \ceil{\frac{1}{\alpha}\log\left(\frac{1}{\varepsilon}\right)}$

\end{proof}


\section{Spectrum}\label{cpsec}

Here we provide a formula for the characteristic polynomial of the commuting chain on a CA group, the proof is in the appendix \Cref{CPProof}.

\begin{theorem}\label{CACP}

Let $G$ be a finite CA group with center of size $z$ and $j$ (distinct) centralizers of size $c_i$ for $1\leq i\leq j$. The characteristic polynomial for the transition matrix of the commuting chain on $G$ is then

\[\lambda^{n-j-1}\left(\frac{z}{|G|}\left(1+\sum_{k=1}^{j}\frac{c_k-z}{c_k(\lambda-1)+z}\right)-\lambda\right)\prod_{i=1}^{j}\left(\lambda-\frac{c_i-z}{c_i}\right)\]

\end{theorem}


\section{Examples}\label{ex}

\subsection{Heisenberg Group}

The Heisenberg group (denoted $H_3(p)$) is the set of $3\times 3$ matrices of the form

\[\begin{pmatrix} 
1 & a & c \\
0 & 1 & b\\
0 & 0 & 1\\
\end{pmatrix}\hspace{1cm} a,b,c\in\zp\]
 
with standard matrix multiplication. The order of $H_3(p)$ is $p^3$. To each element of $H_3(p)$ we associate the ``natural" ordered triple in $(\mathbb{Z}/p\mathbb{Z})^3$ (i.e $(a,b,c)$ in the above). So $(a,b,c)(a',b',c')=(a+a',b+b',c+c'+ab')$.\\

The center of $H_3(p)$ is $\{(0,0,c):c\in\zp\}$.\\

If $X=(a,b,c)$ with $(a,b)\neq (0,0)$ then $C_X=\{(ka,kb,c'):k,c'\in\zp\}$. To see this suppose $(a',b',c')$ commutes with $X$ then

\[(a,b,c)(a',b',c')-(a',b',c')(a,b,c)=(0,0,ab'-a'b)=0\]
so we need $ab'-a'b=0$. Equivalently that $a(a')^{-1}=b(b')^{-1}$. Thus a non-central element of $H(p)$ commutes with $p^2$ elements. This also shows that $H_3(p)$ is a CA group since for non-central $X$ we see $C_X=\{(ka,kb,c'):k,c'\in\zp\}$ is an abelian subgroup.\\
\indent $H_3(p)$ has $p^2+p-1$ conjugacy classes (see e.g ~\cite{Terras}). Take $C$ to be the centralizer of a non-central element, from the orbit stabilizer lemma we know that the size of the conjugacy classes of elements in $C\setminus Z$ is $\frac{p^3}{p^2}=p$. we see $\pi(C\setminus Z)=\frac{1}{p^2+p-1}\frac{1}{p}\left(p^2-p\right)=\frac{p-1}{p^2+p-1}$ which is less than $\frac{1}{2}$. So  the hypotheses of \Cref{LB} and \Cref{cutoff} are met\\
\indent Applying theorem \Cref{minUB} says $t_{mix}\leq p\log{\left(16\right)}+2$. But since all the centralizers are the same size \Cref{UCAmixing} does better with $t_{mix}\leq p\log{\left(4\right)}$. Below we'll see that $\lams\rightarrow 1$ as $p\to\infty$ so we arrive at

\begin{theorem}\label{h3mix}

For the commuting chain on  $H_3(p)$ we have 

\[\frac{p}{4} \leq t_{mix}\leq \ceil{p\log\left(4\right)} \]

furthermore the chain does not present a cutoff. 
\end{theorem}

 \Cref{CACP} tells us the characteristic polynomial is 

 
 \[\lambda^{p^3-p-2}\left(\lambda-1\right)\left(\lambda-\left(1-\frac{1}{p}\right)\right)^p\left(p^2\lambda+p-1\right)\frac{1}{p^2}\]

We list the eigenvalues with multiplicity below ordered in decreasing value. 
\begin{center}
 \begin{tabularx}{0.4\textwidth} {  | >{\centering\arraybackslash}X 
  | >{\centering\arraybackslash}X | }
 \hline
 $\lambda$ & multiplicity \\
 \hline
 1  & 1   \\
  \hline
 $1-\frac{1}{p}$ & $p$ \\
\hline
$0$ & $p^3-p-2$\\
 \hline
$\frac{1-p}{p^2}$ & 1 \\
 \hline

\end{tabularx}
\end{center}

\begin{remark}
Using the bound eigenvalue bound from \Cref{StdEigBnd} gives
\[\left(p-1\right)\log{\left(\frac{1}{2\varepsilon}\right)}\leq \tmixe \ceil{\leq p\log{\left(\frac{p^3+p^2-p}{\varepsilon}\right)}}.\]

This upper bound would not have been useful for disproving cutoff. 

\end{remark}

\subsection{Affine group}\label{CPPoof}

\par  The Affine group, $A(p)$, is the set of $2\times 2$ matricies of the form 
\[\begin{bmatrix}
a & b  \\
0 & 1\\

\end{bmatrix} \hspace{1cm} a,b\in\mathbb{Z}_p,a\neq0\]
 with standard matrix multiplication for prime $p$. The order of $A(p)$ is $p(p-1)$. We use the same shorthand as in the Heisenberg group (so $(a,b)$ for the matrix above). Thus $ (x,y)(x',y')=(aa',ab'+b) $. The center is the identity $(1,0)$. 

If $X=(1,b)$ for $b\neq0$ then $ YX=XY$ iff $Y=(1,b')$ for $b'\in\zp$.\\

If $X=(a,b)$, $a\neq1$ then $XY=YX$ iff $Y=(k(a-1)+1,kb)$ for some $k\in(\zp)\setminus\{-(a-1)^{-1}\}$.\\

The center is of order $1$, there are $p$ centralizers of size $p-1$, and 1 centralizer of size $p$. $A(p)$ has $p$ conjugacy classes (\cite{Terras}). Take $C$ to be the centralizer of size $p$. Applying the orbit stabilizer lemma  shows that the non-central elements of $C$ belong to conjugacy classes of size $p-1$. Then   \[\pi(C\setminus Z)=\frac{p-1}{(p-1)p}\leq \frac{1}{2}.\] 

We'll see that $\lams\rightarrow 1$ as $p\to\infty$ so combining \Cref{LB,minUB} and \Cref{cutoff} we have the following. 
\begin{theorem}\label{affmix}

For the commuting chain on  $A(p)$ we have 

\[\frac{p}{4} \leq t_{mix}\leq \ceil{p\log\left(16\right)+2} \]
furthermore the chain does not present a cutoff.

\end{theorem}

 \Cref{CACP} tells us the characteristic polynomial simplifies to  

\[\lambda^{p^2-2p-2}\frac{(\lambda-1) \left(\lambda-\frac{p-2}{p-1}\right)^{p-1} \left(p^2 \lambda^2-p^2 \lambda-p \lambda^2+3 p \lambda-p-2 \lambda+2\right)}{p }.\]

Our eigenvalues ordered in decreasing value are

\begin{center}
 \begin{tabularx}{0.8\textwidth} { 
  | >{\centering\arraybackslash}X 
  | >{\centering\arraybackslash}X | }
 \hline
 $\lambda$ & multiplicity \\
 \hline
 1  & 1   \\
  \hline
 $\frac{p^2+\sqrt{p^4-2 p^3+p^2-4 p+4}-3 p+2}{2 (p-1) p} $ & $1$ \\
\hline
$\frac{p-2}{p-1}$ & $p-1$\\
 \hline
$0$ & $p^2-2p-2$ \\
 \hline
$\frac{p^2-\sqrt{p^4-2 p^3+p^2-4 p+4}-3 p+2}{2 (p-1) p}$ & $1$ \\
 \hline
\end{tabularx}
\end{center}

\begin{remark}
Similar to the $H_3(p)$ using only $\lambda_\star$ would give an upper bound for $t_{mix}$ insufficient for disproving cutoff.

\end{remark}

\subsection{$GL(2,q)$ -$q$ a power of a prime}

Take $p$ an odd prime and $q=p^k$ for some $k$. Let $GL(2,q)$ be the group of invertible $2\times2$ matrices with entries in the field of order $q$.  Then $GL(2,q)$ is an CA group of order $(q^2-1)(q^2-q)$ with center of size $q-1$ (see \cite{noncompap,speccommgraph}). There are 
\begin{itemize}
\item$\frac{q(q+1)}{2}$ centralizers of size $(q-1)^2$
\item $\frac{q(q-1)}{2}$ centralizers of size $q^2-1$, 
\item $q+1$ centralizers of size $q(q-1)$. 
\end{itemize}

There are $q^2-1$ conjugacy classes. Take $C$ to be a centralizer of size $q^2-1$, the non-central elements belong to conjugacy classes of size $\frac{(q^2-1)(q^2-q)}{q^2-1}=q(q-1)$

\[\pi\left(C\setminus Z\right)=\frac{1}{q^2-1}\frac{q^2-1-(q-1)}{q(q-1)}=\frac{1}{q^2-1}\leq\frac{1}{2}.\] 

Though we can not be sure of an explicit formula for $\lams$ below we see that $\lams\rightarrow 1$ as $q\to\infty$ so we arrive at 

\begin{theorem}\label{glmix}
For the commuting chain on  $GL(2,q)$ with $q$ a power of an odd prime we have 

\[\frac{q+1}{4}\leq\tmix\leq\ceil{\left(q+1\right)\log{\left(16\right)+2}} \]

furthermore the chain does not present cutoff. 

\end{theorem}

From \Cref{CACP} we have that the characteristic polynomials is (after some simplification)

\[\lambda^{(q^4-q^3-2 q^2-2)} (\lambda-1) \left(\lambda-\frac{q-2}{q-1}\right)^{\frac{q}{2}(q+1)-1} 
\left(\lambda-\left(\frac{q-1}{q}\right)\right)^q \left(\lambda-\left(\frac{q}{q+1}\right)\right)^{\frac{q}{2} (q-1) -1} \cdot\]
\[\frac{q^3 (\lambda-1)^2 \lambda+q^2 \left(3 \lambda^2-4 \lambda+1\right)-q \left(\lambda^3-2 \lambda^2-2 \lambda+2\right)+\lambda (3-2 \lambda)}{q}.\]

The roots of the third term (the cubic) do not have a simple expression in terms standard functions, however we conjecture (based on numerics) that $\lams=1-\frac{1}{q+1}$.

\subsection{$PSL(2,2^k)$}

For $k\geq 2$ $PSL(2,2^k)$ is the quotient of $SL(2,2^k)$ ($2\times 2$ matrices with determinate $1$ over a field with $2^k$ elements) with the subgroup of scalar matrices. $PSL(2,2^k)$ has order $2^k(2^{2k}-1)$, has center of order $1$, and is a CA group (\cite{noncompap,speccommgraph} ).
There are 
\begin{itemize}
\item$2^k+1$ centralizers of size $2^k$
\item $2^{k-1}(2^k+1)$ centralizers of size $2^k-1$, 
\item $2^{k-1}(2^k-1)$ centralizers of size $2^k+1$. 
\end{itemize}

From orbit stabilizer we see a conjugacy class corresponding to a centralizer of size $2^k+1$ has size $\frac{2^k(2^{2k}-1)}{2^k+1}=4^k-2^k$. There are $2^k+1$ conjugacy classes for this group. Taking $C$ to be a centralizer of size $2^{k+1}$ we see

\[\pi\left(C\setminus Z\right)=\frac{1}{2^k+1}\frac{2^k+1-1}{4^k-2^k}=\frac{1}{4^k-1}\leq \frac{1}{2}.\]

Like in the case of $GL(2,q)$ we don't have a formula for $\lams$, but we do know it satisfies $\lams\rightarrow 1$ as $k\to\infty$. So we arrive at 

\begin{theorem}\label{pslmix}

For the commuting chain on  $PSL(2,2^k)$ with $k\geq 2$ an integer we have  

\[\frac{2^k+1}{4}\leq\tmix\leq\ceil{\left(2^k+1\right)\log{\left(16\right)}+2}\]

furthermore the chain does not present cutoff. 

\end{theorem}

From \Cref{CACP} we have that the characteristic polynomials is (after some simplification)
\[2^{-k} (\lambda -1) \lambda ^{-2^{k+1}-4^k+8^k-2}\cdot \]
\[\left(\lambda -\frac{2^k-2}{2^k-1}\right)^{2^{k-1} \left(2^k+1\right)-1} 
\left(\lambda -\frac{ 2^k-1}{2^k}\right)^{2^k} 
\left(\lambda -\frac{2^k}{2^k+1}\right)^{2^{k-1}(2^k-1) -1}
\left(2^k-1\right) \cdot\]
\[ \left(-2 \lambda ^2-2^k \lambda ^3+2^{3 k} \lambda ^3+3\ 2^{2 k} \lambda ^2+2^{k+1} \lambda ^2-2^{3 k+1} \lambda ^2+2^{3 k} \lambda +2^{k+1} \lambda -2^{2 k+2} \lambda +3 \lambda +2^{2 k}-2^{k+1}\right) .\]

Like $GL(2,q)$ the third term in the characteristic polynomial does not lead to simple formulas for roots. But we Conjecture $\lams=1-\frac{1}{2^k+1}$.

\subsection{$D_{2n}$}
$D_{2n}$ is a CA group which has $2n$ elements which are $\{1,r ,r^2, \hdots, r^{n-1},s, sr,\hdots, sr^{n-1}\}$. The group is defined be the relations $r^n=s^2=(sr)^2=1$. The conjugacy structure is determined by whether $n$ is odd or even.

\subsubsection{n odd}

When $n$ is odd the center of $D_{2n}$ is just the identity, we have $1$ centralizer of size $n$ and $n$ centralizers of size $2$. There are a total of $\frac{n+3}{2}$ conjugacy classes. \\
\indent If $C$ is the centralizer of size $n$ then  $\pi(C\setminus Z)=\frac{2}{n+3}\frac{1}{2}\left(n-1\right)$. Since $\frac{n-1}{n+3}\geq \frac{1}{2} $ for $n\geq 5$ the hypothesis of \Cref{LB} fail. If we take $C$ to be any of the centralizers with two elements then $\pi(C\setminus Z)=\frac{2}{n+3}\frac{1}{n}$, which is smaller than $\frac{1}{2}$. Then \Cref{LB} gives us the trivial lower bound of $t_{mix}\geq\frac{1}{2}$. 
\indent Notice \Cref{minUB} tells us $t_{mix}(\varepsilon)\leq \ceil{2n\log{\frac{1}{\varepsilon}}+2}$, we believe this to be a poor bound. In the appendix  (\Cref{ccchain,D2nodd}) we show when the chain is started randomly on a fixed conjugacy class the mixing time is bounded by a constant independent of $n$.


For the characteristic polynomial we then have 

\[\frac{(\lambda -1) \left(\lambda -\frac{1}{2}\right)^{n-1} \lambda ^{n-2} \left(2 \lambda +4 \lambda ^2 n-2 \lambda  n-n+1\right)}{4 n}.\]

We list our eigenvalues with multiplicity.

\begin{center}
 \begin{tabularx}{0.8\textwidth} { 
  | >{\centering\arraybackslash}X 
  | >{\centering\arraybackslash}X | }
 \hline
 $\lambda$ & multiplicity \\
 \hline
 1  & 1   \\
  \hline
 $\frac{1}{4} \left(\frac{\sqrt{5 n^2-6 n+1}}{n}-\frac{1}{n}+1\right)$,& $1$ \\
\hline
$\frac{1}{2}$ & $n-1$ \\
 \hline
$0$ & $n-2$ \\
 \hline

 $\frac{1}{4} \left(-\frac{\sqrt{5 n^2-6 n+1}}{n}-\frac{1}{n}+1\right)$ & $1$\\
 \hline
\end{tabularx}
\end{center}

\subsubsection{n even}

When $n$ is even the center of $D_{2n}$ is of order $2$, we have $1$ centralizer of size $n$ and $\frac{n}{2}$ centralizers of size $4$. There are a total of $\frac{n+6}{2}$ conjugacy classes. 
\indent Now if $C$ is the centralizer of size $n$ then  $\pi(C\setminus Z)=\frac{2}{n+6}\frac{1}{2}\left(n-2\right)$. Since $\frac{n-2}{n+6}\geq \frac{1}{2} $ for $n\geq 10$ this is not an acceptable choice to apply \Cref{LB} with. If we take $C$ to be any of the centralizers with four elements then $\pi(C\setminus Z)=\frac{2}{n+6}\frac{2}{n}$, which is smaller than $\frac{1}{2}$. Then \Cref{LB} gives us the trivial lower bound of $t_{mix}\geq\frac{1}{4}$. \\
\indent Similarly to the odd case our general bounds do not do well here.  \Cref{minUB} tells us $t_{mix}(\varepsilon)\leq \ceil{n\log{\frac{1}{\varepsilon}}+2}$, however calculations in the appendix (\Cref{ccchain,D2neven}) do better.


For the characteristic polynomial we have

\[\frac{(\lambda-1) \left(\lambda-\frac{1}{4}\right)^{n-1} \lambda^{n-2} \left(n \left(3 \lambda^2-\lambda-1\right)+3 \lambda+2\right)}{n }. \]

We list our eigenvalues with multiplicity.

\begin{center}
 \begin{tabularx}{0.8\textwidth} { 
  | >{\centering\arraybackslash}X 
  | >{\centering\arraybackslash}X | }
 \hline
 $\lambda$ & multiplicity \\
 \hline
 1  & 1   \\
  \hline
$\frac{n-3+\sqrt{13n^2-30n+9}}{6n}$,& $1$ \\
\hline
$\frac{1}{4}$ & $n-1$ \\
 \hline
$0$ & $n-2$ \\
 \hline

 $\frac{n-3-\sqrt{13n^2-30n+9}}{6n}$ & $1$\\
 \hline
\end{tabularx}
\end{center}

\section{Remarks}\label{re}
The ratio $\frac{c_\star}{z}$ occurs in several places.\begin{itemize}
\item the minorization upper bound -\Cref{minUB}
\item the coupling upper bound  -\Cref{UCAmixing}
\item as part of a potential eigenvalue - \Cref{CACP}
\item as an upper bound to the Cheegar constant \Cref{LB} and \Cref{cutoff}.

\end{itemize}

The Dihedral group example shows that $\frac{c_\star}{z}$ does not always control the mixing time, here $c_\star$ is too big and fails the hypotheses of \Cref{LB}. Furthermore here the term corresponding to $\frac{c_\star-z}{c_\star}$ cancels out of the characteristic polynomial and so $1-\frac{z}{c_\star}$ is not an eigenvalue. If we meet the hypotheses of \Cref{cutoff} we know that mixing is controlled by $\frac{c_\star}{z}$, but $\frac{c_\star}{z}$ may not be the absolute spectral gap, as is the case with the affine group. \\
\indent We conjecture that for $D_{2n}$  that $\tmixn$ is bounded by a constant independent of $n$. Work in the appendix shows that this is the case when the chain has an initial distribution uniform on a given conjugacy class. $D_{2n}$ is the only example we have found where the spectral gap of the chain does not tend to $0$, this led us to the following conjecture.

\begin{conjecture}
For the commuting chain on a CA group of order $n$ we have that $t^{(n)}_{mix}\left(\varepsilon\right)$ is bounded by a constant independent of n if and only if the spectral gap of the chain do not tend to $0$ as $n\to\infty$. 
\end{conjecture}

A potential line of proof is to use the well known fact that for a sequence of reversible ergodic chains with transition matrices $P_n$, equilibrium distributions $\pi_n$, and absolute second largest eigenvalues $\lambda_\star^{(n)}$ there exists a constants $C_n$ such that 

\[\Vert P_n^t(x,\cdot)-\pi_n \Vert_{TV}\leq C_n\left(\lambda^{(n)}_\star\right)^t \]

for all $x$. If one can prove that $C_n$ can be bounded by a constant independent of $n$ then the conjecture would follow. One can formulate expressions for $C_n$ in terms of eigenfunctions of transition matrices. This line of proof is carried out for $D_{2n}$ in the appendix in \Cref{D2nodd,D2neven}. \\

\begin{appendices}

\section{Proof of \Cref{CACP}}\label{CPProof}

\begin{proof}

Let $n=|G|$ we write the $n\times n$ transition matrix in block form as

\[P= \begin{bmatrix} 
W & V \\
X & Z 
\end{bmatrix}\]

 where $W$ is a $z\times z$ matrix with all entires equal to $\frac{1}{n}$.\\
 
$V$ is a $z \times \left(n-z\right)$ matrix also with all entires $\frac{1}{n}$.\\

$X$ is $\left(n-z\right) \times z$, the entries are constant along columns with the first $c_1-z$ rows all having entries being $\frac{1}{c_1}$ while the next $c_2-z$  have all entries equal to $\frac{1}{c_2}$ and so on.\\

$Z$ is a $\left(n-z\right) \times \left(n-z\right)$ block diagonal matrix - the matrices on the diagonal are blocks of size  $\left(c_i-z\right)\times \left(c_i-z\right)$ with all entires $\frac{1}{c_i}$ for $1\leq i \leq j$. The rest of the matrix is all zeros.\\

To calculate $\det(P-\lambda I)$ we note (for appropriately sized $I$)\\

\[\begin{bmatrix}
I & -VZ^{-1}\\
0& I

\end{bmatrix}
\begin{bmatrix}
W & V\\
X& Z

\end{bmatrix}=
\begin{bmatrix}
W-VZ^{-1}X & 0\\
X& Z

\end{bmatrix}\]

and so 
\[\det(P-\lambda I)=\det(Z-\lambda I )\det(\left(W- \lambda I\right)-V(Z-\lambda I)^{-1}X).\]

We now calculate the two terms in the product.\\

\begin{enumerate}

\item $\det(Z-\lambda I )$

Since $Z$ is a block diagonal matrix we just need to know the eigenvalues of the blocks. Each block is $\left(c_i-z\right) \times \left(c_i-z\right)$ with all entires $\frac{1}{c_i}$ for $1\leq i \leq j$. 

Such a matrix has eigenvalue $0$ with multiplicity $c_i-z-1$ and eigenvalue $\frac{c_i-z}{c_i}$ and so 

\[\det(Z-\lambda I )=(-1)^{n-z}\prod_{i=1}^j\left(\lambda-\frac{c_i-z}{c_i}\right)\lambda^{c_i-z-1}\]
\[=(-1)^{n-z}\lambda^{n-z-j}\prod_{i=1}^j\left(\lambda-\frac{c_i-z}{c_i}\right)\]

\item $\det(\left(W- \lambda I\right)-V(Z-\lambda I)^{-1}X)$\\

We'll first calculate each term in the determinant separately. 

\begin{enumerate}

\item $(Z-\lambda I)^{-1}$
Since $Z$ is invertible we can use the Neumann series identity
\[(Z-\lambda I)^{-1}=-\frac{1}{\lambda}(I-\frac{1}{\lambda}Z)^{-1}=-\frac{1}{\lambda}\sum_{k=0}^\infty\left( \frac{1}{\lambda}Z\right)^k\]

$Z^k$ is a block diagonal matrix with $\left(c_i-z \right)\times \left(c_i-z\right)$ blocks with all entires $\frac{(c_i-z)^{k-1}}{c_i^k}$  for $1\leq i \leq j$.

 So $-\frac{1}{\lambda}\sum_{k=0}^\infty\left( \frac{1}{\lambda}Z\right)^k$ is block diagonal with $\left(c_i-z\right) \times \left(c_i-z\right)$ blocks with all entries 
 
 \[\frac{-1}{\lambda}\sum_{k=0}^\infty\frac{(c_i-z)^{k-1}}{\left(\lambda c_i\right)^k}=\frac{-c_i}{\left(c_i-z\right)\left(c_i\left(\lambda-1\right)+z\right)}\]

for $1\leq i \leq j$.

\item $V(Z-\lambda I)^{-1}$\\

$V(Z-\lambda I)^{-1}$ is a $z\times \left(n-z\right)$ matrix constant along rows, the first $c_1-z$ columns have all entires $\frac{-c_1}{n\left(c_1\left(\lambda-1\right)+z\right)}$

\item $V(Z-\lambda I)^{-1}X$\\

 $V(Z-\lambda I)^{-1}X$ is a $z\times z$ matrix with all entries $\sum_{i=1}^j\frac{-\left(c_i-z\right)}{n\left(c_i
\left(\lambda-1\right)+m\right)}=:s$

\end{enumerate}

$\left(W- \lambda I\right)-V(Z-\lambda I)^{-1}X$ has all it's diagonal entries as $\frac{1}{n}-\lambda-s$ while the remaining entries are $\frac{1}{n}-s$. We use some row operations to calculate the determinate. \\

To the first row we add each row beneath now the first row is $m\left(\frac{1}{n}-s\right)-\lambda$. Now to each of the rows below the first we subtract the first - now all rows below the first are all $0$ except for $-\lambda$ on the diagonal.\\

We conclude \[\det(\left(W- \lambda I\right)-V(Z-\lambda I)^{-1}X)=\left(z\left(\frac{1}{n}-s\right)-\lambda\right)\left(-\lambda\right)^{z-1}\]

and from here we're done after noting powers of $-1$ can be disregarded.

\end{enumerate}

\end{proof}

\section{Chain On Conjugacy Classes}
\quad Given a finite group $G$ let $P$ be the transition probabilities for the commuting chain on $G$, we call the Markov chain with transition probabilities given by 

\[\tilde{P}(O_1,O_2)=\sum_{y\in O_2}\sum_{x\in O_1}P(x,y)\frac{1}{|O_1|} \text{ where }O_1,O_2\in\text{Cl}_G(G) \]
the  commuting chain on the conjugacy classes of G.

 To relate the mixing times for the commuting chain on $G$ and the chain on the conjugacy classes we need the following lemma.
 
 \begin{lemma}

\label{ALGFACT}
Let $P$ be the transition matrix for the commuting chain on a group $G$. Then for any $x,y,g\in G$ we have

\[P(x,y)=P(gxg^{-1},gyg^{-1}).\]
\end{lemma}

\begin{proof}
\item The claim follows from the fact that $gC_xg^{-1}=C_{gxg^{-1}}$. To see this note for a group element $h$ \[h(gxg^{-1})=(gxg^{-1})h\iff   (g^{-1}hg) x =x(g^{-1}hg) .        \]

So we have $h\in C_{gxg^{-1}}\iff (g^{-1}hg)\in C_{x}$.

\end{proof}

\begin{theorem}\label{ccchain}

Let $P,\pi$ be the transition matrix and stationary distribution for the commuter's chain on a group $G$ and $\tilde{P},\tilde{\pi}$ and the transition matrix and stationary distribution for the commuter's chain on the conjugacy classes of $G$. For any $t\geq1$ we have

\[\Vert \mu_{K}P^t-\pi \Vert_{TV}=\Vert \tilde{P}^t(K,\cdot)-\tilde{\pi} \Vert_{TV}\] where $\mu_{K}$ is the uniform distribution on conjugacy class $K$.

\end{theorem}

\begin{proof}
Let $O_1,\hdots,O_k$ be an enumeration of the conjugacy classes of $G$.

First note $ \mu_{K}P^t(y)=\frac{1}{|K|}\sum_{z\in K} P(z,y)$ is invariant under conjugation since for any $g$ 
\[ \mu_{K}P^t(gyg^{-1})=\frac{1}{|K|}\sum_{z\in K} P^t(z,gyg^{-1})=\frac{1}{|K|}\sum_{z\in K} P^t(g^{-1}zg,y)=\mu_{K}P^t(y)\]

using \Cref{ALGFACT}. Now since $\mu_{K}P^t(y)$ is constant for a given $y$ in conjugacy class $O$ we have $|O|\mu_{K}P^t(y)=\sum_{y\in O}\mu_{K}P^t(y)=\tilde{P}(K,O)$

So 

\[\Vert \mu_{K}P^t-\pi \Vert_{TV}=\frac{1}{2}\sum_{y\in G}\left|\mu_{K}P^t(y)-\pi(y)\right|=\frac{1}{2}\sum_{i=1}^k\sum_{y\in O_i}\left|\mu_{K}P^t(y)-\pi(y)\right| \]

\[=\frac{1}{2}\sum_{i=1}^k|O_i|\left|\mu_{K}P^t(y)-\pi(y)\right|=\frac{1}{2}\sum_{i=1}^k\left||O_i|\mu_{K}P^t(y)-\tilde{\pi}(O_i)\right|\]
\[=\frac{1}{2}\sum_{i=1}^k\left|\tilde{P}(K,O_i)-\tilde{\pi}(O_i) \right|=\Vert \tilde{P}^t(K,\cdot)-\tilde{\pi} \Vert_{TV} \]

which concludes the proof.
\end{proof}

\section{$D_{2n}$ for odd $n$}

For basic information about $D_{2n}$ see e.g \cite{DF}, in particular we follow the notation established in \cite{DF}. We have $\frac{n+3}{2}$ conjugacy classes, they are $\{1\}$,$ \{r^i,r^{-i}\}$, and $\{s,sr,....sr^{n-1}\}$, where $1\leq i\leq\frac{n-1}{2}$. We will enumerate in that order, that is we will associate with each conjugacy class a natural number as follows:\\

\begin{center}
  \begin{tabular}{ |c | c | c | c | c | c|} 
    \hline
    $1$ & $2$ & $3$ & ... &$\frac{n+1}{2}$ & $\frac{n+3}{2}$\\ \hline
    $\{1\}$ &  $\{r,r^{-1}\}$&  $\{r^2,r^{-2}\}$ &...&$\{r^\frac{n-1}{2}$,$r^{-\frac{n-1}{2}}\}$& $\{s,s^r,...,sr^{n-1}\}$\\ \hline
  \end{tabular}
\end{center}

\quad Let $m=\frac{n+3}{2}$. Our $m\times m$ transition matrix for the commuter's chain on conjugacy classes can then be written as:
\begingroup
\renewcommand*{\arraystretch}{1.5}
\[ P =\begin{pmatrix}
    \frac{1}{2n} & \frac{1}{n} & \dots & \frac{1}{n}  & \frac{1}{2} \\
    
    \frac{1}{n} & \frac{2}{n} & \dots & \frac{2}{n}  & 0\\
    
    \vdots & \vdots & \ddots & \vdots & \vdots \\
     
     \frac{1}{n} & \frac{2}{n} & \dots & \frac{2}{n}  & 0\\
  
     \frac{1}{2}&0& 0 & \dots  & \frac{1}{2}

\end{pmatrix}
\]
\endgroup

with stationary distribution $\pi(i)=\frac{2}{n+3}$ for $1\leq i \leq m$.\\

 We will list the eigenvalues and eigenvectors corresponding to the commuting chain on conjugacy classes of $D_{2n}$ for odd $n$ after this lemma. 

\begin{lemma}\label{ZEROMULT}
Suppose $A$ is a $n\times n$ symmetric matrix with $k$ identical rows. Then $0$ is an eigenvalue of $A$ with multiplicity at least $k-1$.
\end{lemma}
\begin{proof}
\quad Since $A$ is symmetric there exists an orthogonal similarity transformation into a diagonal matrix, say $ A= QDQ^T$ where $D$ is diagonal with the diagonal entries being the eigenvalues of $A$.\\

\quad Since $A$ has $k$ identical rows the rank of $A$ can be no greater than $n-k+1$. Since rank is invariant under multiplication of a matrix of full rank (i.e $Q$) we have that that the rank of $D$ is no greater than $n-k+1$, this is only possible if $0$ appears $n-m+1$ times on the diagonal of $D$. This implies that $0$ is an eigenvalue of $A$ with multiplicity of at least $k-1$. 

\end{proof}

Let \[c_n=\sqrt{(5n-1)(n-1)},\]

\[A_n=\sqrt{ \frac{n^2(n+3)}{5n^2-nc_n+4n+c_n-1 }} \text{, and}\] \\ 
\[B_n=\sqrt{\frac{n^2(n+3)}{5n^2+nc_n+4n-c_n-1  }}. \]
The nonzero eigenvalues for the transition matrix are: 
\[\lambda_1=1,\]
\[\lambda_2=\frac{n-1+c_n}{4n} \text{, and}\] 
 \[\lambda_{m}=\frac{n-1-c_n}{4n}.\]\\

Our corresponding eigenfunctions, normalized in $\ell^2(\pi)$, are\\

 \[ f_1=(1,\dots,1), \]\\
\[ f_2 =A_n\left( \frac{1}{n}\left(-\frac{n+1}{2}+\frac{c_n}{2}\right),\frac{1}{n}\left(-1-\frac{c_n}{n-1}\right),\dots, \frac{1}{n}\left(-1-\frac{c_n}{n-1}\right),1\right),
  \]\\
  and
  
  \[ f_{m} =B_n \left( \frac{1}{n}\left(-\frac{n+1}{2}-\frac{c_n}{2}\right),\frac{1}{n}\left(-1+\frac{c_n}{n-1}\right),\dots, \frac{1}{n}\left(-1+\frac{c_n}{n-1}\right),1\right). \]\\

That these are indeed eigenvalues and functions can be verified with a computation.

 \Cref{ZEROMULT} assures us that these are the only non-zero eigenvalues. To show that $\tmix$ is bounded by a constant (independent of $n$) we need the following lemma which is easy to check.

\begin{lemma}\label{d2n0ddbounds}

For the eigenvalues and functions above we have for $n\geq3$

\begin{enumerate}

\item $n\leq c_n\leq 3n$

\item $A_n\leq\sqrt{n}$ 

\item $B_n\leq\sqrt{n}$

\item$ |\frac{f_2(i)}{A_n}|\leq \frac{\sqrt{5}+1}{2}$ for any $i$

\item $ |\frac{f_m(i)}{B_n}|\leq \frac{\sqrt{5}+1}{2}$ for any $i$

\item $ |\frac{f_2(i)}{A_n}|\leq \frac{4}{n}$ for $2\leq i\leq m-1$

\item $ |\frac{f_m(i)}{B_n}|\leq \frac{4}{n}$ for $2\leq i\leq m-1$.

\item $\lambda_\star<\frac{1+\sqrt{5}}{4}$

\end{enumerate}

\end{lemma}

\begin{theorem}\label{D2nodd}

For the commuting chain on the conjugacy classes of $D_{2n}$ for odd $n\geq3$ we have for any $\varepsilon\in(0,1)$ there exists a $C_{\varepsilon}$ (independent of n)
\[t_{mix}^{(n)}\left(\varepsilon\right)\leq C_{\varepsilon}.\]

\end{theorem}

\begin{proof}

Let $d^{(n)}(t)$ be the distance to stationarity for the commuting chain on the conjugacy classes of $D_{2n}$. We show that for some $\delta\in(0,1) $ we have $d^{(n)}(t)\leq C(\lambda_\star^{(n)})^t\leq C\delta^t=C(1-(1-\delta))^t\leq Ce^{-(1-\delta)t}$  for some $C$ for all $n$. Then taking $t_{mix}^{(n)}(\varepsilon)\leq\frac{1}{1-\delta}\log{\frac{C}{\varepsilon}}$ does the trick as $1-\delta$ is bounded away from $0$.\\

Using the spectral decomposition for a reversible transition matrix we have the following for a fixed starting state $i$ (note here $P$ and $\pi$ are for the chain on the conjugacy classes) 
\[\Vert P^t(i,\cdot)-\pi \Vert_{TV}=  \frac{1}{2}\sum_{x\in\mathcal{X}}| P^t(i,x)-\pi(x)|  = \frac{1}{2}\sum_{x\in\mathcal{X}}|\pi(x)\sum_{j=2}^{m}f_j(i)f_j(x)\lambda^t_j|\]

\[\leq\frac{1}{2}\sum_{x\in\mathcal{X}}\pi(x)\sum_{j=2}^{m}|f_j(i)||f_j(x)||\lambda^t_j|=\frac{1}{n+3}\sum_{x\in\mathcal{X}}\sum_{j=2}^{m}|f_j(i)||f_j(x)||\lambda^t_j|\]
\[\leq \frac{1}{n+3}\lambda_{2}^t\sum_{x\in\mathcal{X}}\sum_{j=2}^{m}|f_j(i)||f_j(x)|=
\frac{1}{n+3}\lambda_{2}^t\sum_{x\in\mathcal{X}}\left(|f_2(i)||f_2(x)|+|f_m(i)||f_m(x)|\right) \]

\[=\frac{1}{n+3}\lambda_{2}^t\left(2n\left(\frac{\sqrt{5}+1}{2}\right)^2 +\sum_{x\in\mathcal{X}\setminus \{1,m\}}\left(|f_2(i)||f_2(x)|+|f_m(i)||f_m(x)|\right)  \right)\]

\[\leq\frac{1}{n+3}\lambda_{2}^t\left(2n\left(\frac{\sqrt{5}+1}{2}\right)^2 + \frac{n-1}{2}\frac{4}{n}\frac{\sqrt{5}+1}{2} n \right)\]

\[\leq\frac{n}{n+3}\left(\frac{\sqrt{5}+1}{2}\right)^2 \lambda_{2}^t\left(2 + \frac{n-1}{2}\frac{4}{n} \right)\]

\[\leq4 \frac{n}{n+3} \left(\frac{\sqrt{5}+1}{2}\right)^2\lambda_{2}^t  \leq  4  \left(\frac{\sqrt{5}+1}{2}\right)^2\lambda_{2}^t \]

using that $\frac{n-1}{n}<1$, $\frac{n}{n+3}<1$, $\left(\frac{\sqrt{5}+1}{2}\right)<\left(\frac{\sqrt{5}+1}{2}\right)^2$, $\lams=\lambda_2$, and the bounds of \Cref{d2n0ddbounds}. Since the above inequality holds for all starting states $i$ we have that $d^{(n)}(t)\leq C\lams^t$. Since $\lams<\frac{1+\sqrt{5}}{4}<1$ we are done.

\end{proof}

\section{$D_{2n}$ for even $n$}

As in the odd case we follow the notation established in \cite{DF}. We have $\frac{n+6}{2}$ conjugacy classes, they are $\{1\},\{r^{\frac{n}{2}}\}, \{r^i,r^{-i}\},\{s,sr^2,....sr^{n-2}\}, ,$ and $\{sr,....sr^{n-1}\}$, where $3\leq i\leq\frac{n+2}{2}$. We will enumerate in that order, that is we will associate with each conjugacy class a natural number as follows:\\

\begin{center}
  \begin{tabular}{ |c | c | c | c | c | c| c| c|} 
    \hline
    $1$ & $2$ & $3$ & ...  &$\frac{n+2}{2} $ & $\frac{n+4}{2} $& $\frac{n+6}{2}$\\ \hline
    $\{1\}$ & $\{r^{\frac{n}{2}}\}$&  $\{r,r^{-1}\}$ &...&$\{r^\frac{n-2}{2}$,$r^{-\frac{n-2}{2}}\}$&$\{s,sr^2,\ldots,sr^{n-2}\}$ &$\{sr,...,sr^{n-1}\}$\\ \hline
  \end{tabular}
\end{center}

Let $m=\frac{n+6}{2}$. When $\frac{n}{2}$ is odd our transition matrix can be written as:\\

\begingroup
\renewcommand*{\arraystretch}{1.5}
\[ P =\begin{pmatrix}
    \frac{1}{2n}&\frac{1}{2n} & \frac{1}{n} & \dots & \frac{1}{n}  & \frac{1}{4} & \frac{1}{4}\\
    
          \frac{1}{2n} &\frac{1}{2n}& \frac{1}{n} & \dots & \frac{1}{n}  & \frac{1}{4} & \frac{1}{4}\\

    \frac{1}{n}&\frac{1}{n}  & \frac{2}{n} & \dots & \frac{2}{n}    &0 & 0\\
    
    \vdots & \vdots & \vdots & \ddots & \vdots & \vdots & \vdots \\
     
     \frac{1}{n}&\frac{1}{n} & \frac{2}{n} & \dots & \frac{2}{n}     &0 & 0\\

      \frac{1}{4}& \frac{1}{4}& 0 & \dots &0 & \frac{1}{4}& \frac{1}{4}\\
     
     \frac{1}{4}& \frac{1}{4}& 0 & \dots &0 & \frac{1}{4}& \frac{1}{4}

\end{pmatrix}
\]
\endgroup

with stationary distribution $\pi(i)=\frac{2}{n+6}$ for $1\leq i \leq m$. If $\frac{n}{2}$ is even $\pi$ is the same but the the bottom right corner is replaced with $\frac{1}{2}I_2$ (where $I_2$ is the $2\times 2 $ identity  matrix).\\

We  list the eigenvalues and eigenvectors corresponding to the commuting chain on conjugacy classes of $D_{2n}$ for even $n$ when $\frac{n}{2}$ is odd. We'll note what changes in the case where $\frac{n}{2}$ is even.  To that end let 

\[c_n=\sqrt{(5n-2)(n-2)},\]

\[A_n=\sqrt{ \frac{n^2(n+6)} {2 \left(5n^2-nc_n+8n+2c_n-4 \right) } }  \text{, and}\] \\ 
\[B_n=\sqrt{ \frac{n^2(n+6)} {2 \left(5n^2+nc_n+8n-2c_n-4 \right) } }.\]

The nonzero eigenvalues for the transition matrix are: 
\[\lambda_1=1,\]
\[\lambda_2=\frac{n-2+c_n}{4n} \text{, and}\] 
 \[\lambda_{m}=\frac{n-2-c_n}{4n}.\]\\

Our corresponding eigenfunctions, normalized in $\ell^2(\pi)$, are\\

 \[ f_1=(1,\dots,1), \]\\
\[ f_2 =A_n\left( \frac{1}{n}\left(-\frac{n+2}{2}+\frac{c_n}{2}\right),\frac{1}{n}\left(-\frac{n+2}{2}+\frac{c_n}{2}\right),\frac{2}{n}\left(-1-\frac{c_n}{n-2}\right),\dots, \frac{2}{n}\left(-1-\frac{c_n}{n-2}\right),1,1\right),
  \]\\
  and
  
  \[ f_{m} =B_n \left( \frac{1}{n}\left(-\frac{n+2}{2}-\frac{c_n}{2}\right),\frac{1}{n}\left(-\frac{n+2}{2}-\frac{c_n}{2}\right),\frac{2}{n}\left(-1+\frac{c_n}{n-2}\right),\dots, \frac{2}{n}\left(-1+\frac{c_n}{n-2}\right),1,1\right). \]
  
  If $\frac{n}{2}$ is even the above eigenfunctions and values remain however an eigenvalue of $\frac{1}{2}$ is added with corresponding eigenfunction  $\left(-\frac{1}{2}\sqrt{\frac{n+6}{2}},\frac{1}{2}\sqrt{\frac{n+6}{2}},0,\dots,0\right)$. The remainder of the analysis will be in the case that $\frac{n}{2}$ is odd, the other case follows from (slightly) modified computations.

We need the following easily verified lemma to bound mixing times. 

\begin{lemma}\label{d2nevenbounds}

For the eigenvalues and functions above we have for $n\geq4$ and $\frac{n}{2}$ odd

\begin{enumerate}

\item $|c_n|\leq 4n$

\item $A_n\leq\sqrt{n}$ 

\item $B_n\leq\sqrt{n}$

\item$ |\frac{f_2(i)}{A_n}|\leq1+\sqrt{5}$ for any $i$

\item $ |\frac{f_m(i)}{B_n}|\leq1+\sqrt{5}$ for any $i$

\item $ |\frac{f_2(i)}{A_n}|\leq \frac{16}{n}$ for $3\leq i\leq m-2$

\item $ |\frac{f_m(i)}{B_n}|\leq \frac{16}{n}$ for $2\leq i\leq m-2$

\item $\lams<\frac{1+\sqrt{5}}{4}.$

\end{enumerate}

\end{lemma}

\begin{theorem}\label{D2neven}

For the commuting chain on the conjugacy classes of $D_{2n}$ for even $n\geq4$ and $\frac{n}{2}$ odd we have for any $\varepsilon\in(0,1)$ there exists a $C_{\varepsilon}$ (independent of n)
\[t_{mix}^{(n)}\left(\varepsilon\right)\leq C_{\varepsilon}.\]

\end{theorem}

\begin{proof}

As in the even case we show $d^{(n)}(t)\leq C\lams^t$  for some $C$ for all $n$, where $d^{(n)}(t)$ is the distance to stationarity for the chain on the conjugacy classes of $D_{2n}$.

Using the spectral decomposition for a reversible transition matrix we have the following for a fixed starting state $i$ \[\Vert P^t(i,\cdot)-\pi \Vert_{TV}=  \frac{1}{2}\sum_{x\in\mathcal{X}}| P^t(i,x)-\pi(x)|  = \frac{1}{2}\sum_{x\in\mathcal{X}}|\pi(x)\sum_{j=2}^{m}f_j(i)f_j(x)\lambda^t_j|\]

\[\leq\frac{1}{2}\sum_{x\in\mathcal{X}}\pi(x)\sum_{j=2}^{m}|f_j(i)||f_j(x)||\lambda^t_j|=\frac{1}{n+6}\sum_{x\in\mathcal{X}}\sum_{j=2}^{m}|f_j(i)||f_j(x)||\lambda^t_j|\]

\[\leq \frac{1}{n+6}\lambda_{2}^t\sum_{x\in\mathcal{X}}\sum_{j=2}^{m}|f_j(i)||f_j(x)|=
\frac{1}{n+6}\lambda_{2}^t\sum_{x\in\mathcal{X}}\left(|f_2(i)||f_2(x)|+|f_m(i)||f_m(x)|\right) \]

\[=\frac{1}{n+6}\lambda_{2}^t\left(2n\left(\sqrt{5}+1\right)^2 +\sum_{x\in\mathcal{X}\setminus \{1,2,m-1,m\}}\left(|f_2(i)||f_2(x)|+|f_m(i)||f_m(x)|\right)  \right)\]

\[\leq\frac{1}{n+6}\lambda_{2}^t\left(2n\left(\sqrt{5}+1\right)^2 + \frac{n-2}{2}\frac{16}{n}\left(\sqrt{5}+1\right) n\right)\]

\[\leq\frac{n}{n+6}\lambda_{2}^t  \left(\sqrt{5}+1\right)^2 \left(2 + \frac{n-2}{2}\frac{16}{n} \right)\]

\[\leq 10\frac{n}{n+6}  \left(\sqrt{5}+1\right)^2\lambda_{2}^t \leq 10  \left(\sqrt{5}+1\right)^2 \lambda_{2}^t  \]

using that $\frac{n-2}{n}<1$, $\frac{n}{n+6}<1$, $\left(\sqrt{5}+1\right)<\left(\sqrt{5}+1\right)^2$, $\lams=\lambda_2$, and the eigenvectors bounds of \Cref{d2nevenbounds}. Since the above inequality holds for all starting states $i$ we have that $d^{(n)}(t)\leq C\lams^t$.

\end{proof}

\end{appendices}

\section*{Acknowledgements} We'd like to thank Jason Fulman for suggesting this problem and providing guidance and helpful discussion. We'd also like to thank Persi Diaconis and Robert Guralnick for helpful discussion. Finally we'd like to thank the referee for thoroughly reading through this article and pointing out numerous errors.

\bibliography{bibl} 
\bibliographystyle{alpha}

 \end{document}